\documentclass[a4paper]{amsart}
\usepackage{amsmath,amsthm,amssymb,hyperref,mathrsfs}
\usepackage{aliascnt}
\usepackage{lmodern}
\usepackage[T1]{fontenc}

\usepackage[textsize=footnotesize]{todonotes}

\usepackage{xcolor}
\definecolor{dblue}{rgb}{0,0,0.70}
\hypersetup{
	unicode=true,
	colorlinks=true,
	citecolor=dblue,
	linkcolor=dblue,
	anchorcolor=dblue
}

\makeatletter
\expandafter\g@addto@macro\csname th@plain\endcsname{%
		\thm@notefont{\bfseries}
	}%
\expandafter\g@addto@macro\csname th@remark\endcsname{%
		\thm@headfont{\bfseries}
	}%
\makeatother

\newtheorem{theorem}{Theorem}[section]
\newtheorem*{theorem*}{Theorem}

\newaliascnt{lemma}{theorem}
\newtheorem{lemma}[lemma]{Lemma}
\aliascntresetthe{lemma}
\newtheorem*{lemma*}{Lemma}

\newaliascnt{claim}{theorem}
\newtheorem{claim}[claim]{Claim}
\aliascntresetthe{claim}

\newaliascnt{proposition}{theorem}
\newtheorem{proposition}[proposition]{Proposition}
\aliascntresetthe{proposition}

\newaliascnt{corollary}{theorem}
\newtheorem{corollary}[corollary]{Corollary}
\aliascntresetthe{corollary}

\newaliascnt{fact}{theorem}
\newtheorem{fact}[fact]{Fact}
\aliascntresetthe{fact}

\theoremstyle{remark}

\newaliascnt{remark}{theorem}
\newtheorem{remark}[remark]{Remark}
\aliascntresetthe{remark}
\newaliascnt{question}{theorem}
\newtheorem{question}[question]{Question}
\aliascntresetthe{question}

\newtheorem*{question*}{Question}

\newaliascnt{definition}{theorem}
\newtheorem{definition}[definition]{Definition}
\aliascntresetthe{definition}

\newaliascnt{example}{theorem}

\aliascntresetthe{example}

\renewcommand{\restriction}{\mathbin\upharpoonright}

\newcommand{\axiomft}[1]{\mathsf{#1}}
\newcommand{\ZFC}{\axiomft{ZFC}}

\newcommand{\Ord}{\mathrm{Ord}}

\newcommand{\GCH}{\axiomft{GCH}}
\newcommand{\PFA}{\axiomft{PFA}}

\newcommand{\MRP}{\axiomft{MRP}}
\newcommand{\MA}{\axiomft{MA}}
\newcommand{\MM}{\axiomft{MM}}
\newcommand{\Str}{\axiomft{Str}}

\DeclareMathOperator{\FA}{FA}
\DeclareMathOperator{\cf}{cf}
\DeclareMathOperator{\ot}{ot}
\DeclareMathOperator{\dom}{dom}
\DeclareMathOperator{\range}{range}

\DeclareMathOperator{\Col}{Col}
\DeclareMathOperator{\Add}{Add}

\newcommand{\forces}{\mathrel{\Vdash}}

\newcommand{\compatible}{\mathrel{||}}

\newcommand{\PP}{\mathbb{P}}
\newcommand{\power}{\mathcal{P}}

\newcommand{\QQ}{\mathbb{Q}}

\newcommand{\cD}{\mathcal D}

\newcommand{\cS}{\mathcal S}
\newcommand{\cT}{\mathcal T}
\newcommand{\cK}{\mathcal K}

\newcommand{\tup}[1]{\langle#1\rangle}

\author{David Asper\'o}
\author{Sean Cox}
\author{Asaf Karagila}
\author{Christoph Weiss}
\thanks{The first author acknowledges support of EPSRC Grant [EP/N032160/1]. The second author was supported by the National Science Foundation (DMS-2154141). The third author was supported by the Royal Society Newton International Fellowship, grant no.~NF170989 and UKRI Future Leaders Fellowship [MR/T021705/2]. No data are associated with this article.}
\thanks{The fact that $\omega_1$-$\Str\PFA$ is consistent relative to $\ZFC$ was proved independently by Asper\'{o}-Karagila and Cox-Weiss.  Extensions to larger $\kappa$ are due to Asper\'{o}-Karagila.}

\address[Asper\'o]{School of Mathematics,
University of East Anglia,
Norwich, NR4~7TJ, UK}
\email{d.aspero@uea.ac.uk}

\address[Cox]{Department of Mathematics and Applied Mathematics,
  Virginia Commonwealth University,
  1015 Floyd Avenue, Richmond, Virginia 23284, USA}
\email{scox9@vcu.edu}

\address[Karagila]{School of Mathematics,
  University of Leeds,
  Leeds, LS2~9JT, UK}
\email{karagila@math.huji.ac.il}

\address[Weiss]{WSoptics GmbH,
S\"udliche Keltenstra{\ss}e 3,
86972 Altenstadt,
Germany
}
\email[Christoph Weiss]{weiss@wsoptics.de}
\date{18 March, 2024}


\title{The $\kappa$-Strongly Proper Forcing Axiom}

\begin{document}
\begin{abstract}
We study methods to obtain the consistency of forcing axioms, and particularly higher forcing axioms. We first force over a model with a supercompact cardinal $\theta>\kappa$ to get the consistency of the forcing axiom for $\kappa$-strongly proper forcing notions which are also $\kappa$-lattice, and then eliminate the need for large cardinals. The proof goes through a natural reflection property for $\kappa$-strongly proper forcings. We also produce a model of this forcing axiom with $2^\kappa$ arbitrarily large, and prove the inconsistency of certain natural strengthenings of the axiom.
\end{abstract}
\maketitle
\section{Introduction}
Forcing axioms are set-theoretic axioms which state that the universe is ``rich with filters'' for forcing notions in a particular class. More technically, forcing axioms are statements saying that given a forcing notion in a particular class, and a ``relatively small'' collection of dense open subsets, there is a filter which meets all dense open sets in the collection. Martin's Axiom states that if $\PP$ is c.c.c.\ (i.e., if $\PP$ satisfies the countable chain condition) and $\{D_\alpha\mid\alpha<\kappa\}$ is a family of dense open subsets of $\PP$, where $\kappa<2^{\aleph_0}$, then there is a filter $G$ such that $G\cap D_\alpha\neq\varnothing$ for all $\alpha<\kappa$. $\ZFC$ trivially proves that for $\kappa=\aleph_0$ such filters exist for any forcing $\PP$, regardless of its combinatorial properties.

When we assume forcing axioms hold in the universe $V$, we can prove that there are objects in $V$ which exhibit ``somewhat generic properties''. For instance, if we assume Martin's Axiom, and $\{f_\alpha\mid\alpha<\kappa\}\subseteq\omega^\omega$ for $\kappa<2^{\aleph_0}$, then we may consider $\PP=\omega^{<\omega}$ and $D^\alpha_n=\{s\in\omega^{<\omega}\mid f_\alpha(m)<s(m)\mbox{ for some }m\in |s|,\,m>n\}$, for $\alpha<\kappa$ and $n<\omega$, as our dense open sets. If $G\cap D^\alpha_n\neq\varnothing$ for all $\alpha$ and $n$, then $\bigcup G=g\neq f_\alpha$ for all $\alpha$, and in fact for each $\alpha$, $g(m)>f_\alpha(m)$ holds infinitely often. Therefore, Martin's Axiom implies that $\mathfrak{d}$, the dominating number, equals $2^{\aleph_0}$.

In the classical case, forcing axioms are phrased around $\aleph_1$ as the main cardinal of interest. In this context, the forcing notions themselves somehow revolve around this (e.g., properness is defined with models of size less than $\aleph_1$, i.e., countable). Recent work on extensions of classical forcing axioms such as the Proper Forcing Axiom, relative to collections of more than $\aleph_1$ dense sets, deals with subclasses of proper forcing notions, and tries to push the size of $2^{\aleph_0}$ to $\aleph_3$ or higher. This is  difficult, since amongst these ``somewhat generic properties'' we can find, for example, closed and unbounded subsets of $\omega_2$ contradicting club guessing on $\omega_2\cap \cf(\omega)$,\footnote{The existence, in $\ZFC$, of a club-sequence $\langle C_\alpha\mid\alpha<\omega_2, \cf(\alpha)=\omega\rangle$ guessing clubs of $\omega_2$---i.e., such that every club $C\subseteq\omega_2$ includes some $C_\alpha$---is a well-known result of Shelah.} or functions $\omega_2\to\{0, 1\}$ uniformising colourings for which there is no uniformisation (see \cite{Shelah:PIF}).

Moving to higher cardinals is harder also because we lose our iteration theorems. Iterating c.c.c.\ forcing notions with finite support is still c.c.c., and iterating proper forcing notions with countable support is still proper. But moving to higher cardinals, even if we require the forcings to be very closed, might
result in unwanted cardinal collapsing (see \cite{Shelah:PIF}).

James Cummings, Mirna D\v{z}amonja, and Itay Neeman proved in \cite{CDN:arXiv} the consistency of a forcing axiom of this flavour by replacing c.c.c.\ by a more restrictive form of the $\kappa^+$-c.c.\ known as the strong $\kappa^+$-c.c. In this note we deal with $\kappa$-strong properness, a more general notion than the strong $\kappa^+$-c.c.\ from \cite{CDN:arXiv}.
We show that Neeman's consistency proof of $\PFA$ using finite conditions can be generalised quite easily to this context even when $\kappa$ is uncountable. We then prove that $\kappa$-strongly proper forcings satisfy a weak reflection property: to prove that enough filters exist for any $\kappa$-strongly proper forcing, it is enough to prove that enough filters exist for $\kappa$-strongly proper subforcings of size $2^\kappa$. Using this reflection property, together with an argument involving the fact that all $\kappa$-sequences of ordinals added by a $\kappa$-lattice and $\kappa$-strongly proper forcing come from adding a Cohen subset of $\kappa$ (\autoref{cohen-subsets}), we show that the assumption of a supercompact cardinal (or any large cardinal) is in fact unnecessary. We then modify our construction and show the consistency of the forcing axiom together with $2^\kappa$ being arbitrarily large. This modified construction actually shows the consistency of a slightly stronger form of our forcing axiom and does not need the fact that $\kappa$-sequences of ordinals added by a $\kappa$-lattice and $\kappa$-strongly proper forcing come from a $\kappa$-Cohen extension.

Our main result is thus the consistency relative to $\ZFC$, for any given regular cardinal $\kappa$ with $\kappa^{<\kappa}=\kappa$, of the forcing axiom, for families of $\kappa^+$-many dense sets, for the family of forcing notions which are both $\kappa$-lattice and $\kappa$-strongly proper. This is a rather small class, containing $\kappa$-Cohen forcing and the natural forcing for adding a club of $\kappa^+$ with conditions of size less than $\kappa$, but not much more. One consequence of the corresponding forcing axiom, due to the inclusion in the class of the above forcing for adding  a club of $\kappa^+$, is the failure of tail club-guessing on $\kappa^+$ for ordinals of cofinality $\kappa$; in other words, the forcing axiom implies that for every sequence $\langle C_\alpha\mid \alpha\in \kappa^+,\,\cf(\alpha)=\kappa\rangle$, where each $C_\alpha$ is a club of $\alpha$, there is a club $C\subseteq\kappa^+$ such that $C_\alpha\setminus C$ is unbounded in $\alpha$ for every $\alpha<\kappa^+$ of cofinality $\kappa$.\footnote{When $\kappa=\omega$, the consistency of the above club-guessing failure is of course well-known. For $\kappa>\omega$, the consistency of the corresponding club-guessing failure is due to Shelah (e.g., \cite{Shelah:2003}).} One could try to obtain a more useful forcing axiom by considering a slightly broader class of forcing notions. We finish the paper observing that certain natural moves in this direction lead to inconsistent principles.

Throughout the paper we work in $\ZFC+\GCH$ for the sake of simplicity, although many of these results can be proved without $\GCH$ if one is willing to collapse cardinals, as long as one assumes that $\kappa^{<\kappa}=\kappa$ where appropriate.

The structure of the first part of the paper is what we may call an ``onion proof''. We start by sketching Neeman's consistency proof of $\PFA$ in the present context, using a supercompact cardinal. We then prove the weak reflection lemma, which allows us to ``peel off'' the consistency strength of the proof to a mere inaccessible cardinal, and then we show that this too can be reduced to nothing more than $\ZFC$. In the second part of the paper we modify (in \autoref{section7}) our final construction from the first part so as to obtain a model of our forcing axiom with $2^\kappa$ arbitrarily large, and in \autoref{relax} we prove that various natural strengthenings of the forcing axiom are actually false.

\subsection*{Acknowledgements}
The authors would like to thank Menachem Magidor for a stimulating conversation that helped shape this research. They also thank Corey Switzer for his useful comments, and in particular for asking about the status of the forcing axiom in \autoref{incons} when $\kappa$ is inaccessible.

\section{Preliminaries}

We say that a relation $\prec$ is a weak total order on $X$ if the transitive closure of $\prec$ is a total order on $X$. We say that a set $M$ is \textit{$\kappa$-closed} if for every $\alpha<\kappa$, every function $f\colon\alpha\to M$ is already in $M$. In the case of a forcing $\PP$, we say that $\PP$ is $\kappa$-closed if every decreasing sequence of length less than $\kappa$ has a lower bound, and that $\PP$ is $\kappa$-directed closed if every directed set of size less than $\kappa$ has a lower bound. We will say that a forcing is \textit{$\kappa$-lattice} if every set of size less than $\kappa$ of pairwise compatible conditions has a greatest lower bound.\footnote{It would perhaps be more appropriate to call such a forcing notion a \textit{$\kappa$-lower semi-lattice}. However, in the interest of keeping things simple, we will not use this terminology.}

\subsection{Strong properness}
\begin{definition}
Let $M$ be a set and $\PP$ a forcing in $M$. We say that a condition $q\in\PP$ is \textit{strongly $M$-generic} (for $\PP$) if for every $q'\leq q$ there is $\pi_M(q')\in\PP\cap M$ such that every condition in $\PP\cap M$ extending $\pi_M(q')$ is compatible with $q'$.
\end{definition}
\begin{definition}
Let $\QQ$ be a forcing notion and let $\cK$ be a class of models. We say that $\QQ$ is \textit{strongly proper for $\cK$} if for every cardinal $\chi$ and every $M\in\cK$ such that $M\prec H(\chi)$ and $\QQ\in M$, every $p\in\QQ\cap M$ can be extended to a strongly $M$-generic condition.

When $\cK$ is the class of all $\kappa$-closed models $M$ of size $\kappa$, we simply say that $\QQ$ is $\kappa$-strongly proper.
\end{definition}

The following is a generalisation of an observation of Mitchell in \cite{Mitchell:2006}.

\begin{proposition}\label{cohen-subsets}
Suppose that $\kappa^{<\kappa}=\kappa$, and let $\PP$ be a $\kappa$-lattice and $\kappa$-strongly proper forcing notion. Any $\kappa$-sequence of ordinals added by $\PP$ is $\kappa$-Cohen generic.
\end{proposition}
\begin{proof}
Let $\dot f$ be a $\PP$-name and $p\in\PP$ such that $p\forces\dot f$ is a $\check\kappa$-sequence of ordinals. Let $M$ be a $\kappa$-closed elementary submodel of $H(\chi)$, for some large enough $\chi$, such that $\PP,\dot f,p\in M$. We let $\QQ=\PP\cap M$. By elementarity of $M$ and its $\kappa$-closedness we have that $\QQ$ is a $\kappa$-lattice forcing of size $\kappa$, and therefore by a back-and-forth argument, using $\kappa^{<\kappa}=\kappa$, we have that $\QQ$ is isomorphic to $\kappa^{<\kappa}$.

Let $p^*\leq p$ be a strongly $M$-generic condition. Let $G$ be a $V$-generic filter with $p^*\in G$. Then $G\cap\QQ$ is $V$-generic for $\QQ$, and this is forced by $p^*$: given any $q_0\leq p^*$ and any dense subset $D\subseteq\QQ$, we may extend the projection of $q_0$ into $M$, $\pi_M(q_0)$, to a condition $q\in D$, and since $\QQ\subseteq M$, $q\in M$ and therefore compatible with $q_0$ in $\PP$. Also, for every $\alpha<\kappa$ the set $D_\alpha$ of $\PP$-conditions deciding $\dot f(\check\alpha)$ is in $M$, and by elementarity of $M$, $D_\alpha\cap M$ is a dense subset of $\QQ$. This means that $\dot g$ defined by $\{\tup{q,\tup{\check\alpha, \check\beta}}\mid q\forces_\PP\dot f(\check\alpha)=\check\beta, q\in M\}$ is a $\QQ$-name such that $p^*\forces_\PP\dot f=\dot g$, and therefore $\dot f^G=\dot g^{G\cap\QQ}\in V[G\cap\QQ]$.
\end{proof}
The following is clear.
\begin{proposition}
The $\kappa$-support iteration of $\kappa$-lattice forcings is $\kappa$-lattice.\qed
\end{proposition}

It will be convenient to adopt the following notation: Given a class $\Gamma$ of forcing notions and a cardinal $\lambda$, $\FA_\lambda(\Gamma)$ is the assertion that for every $\PP\in\Gamma$ and every collection $\mathcal D=\{D_\alpha\mid \alpha<\lambda\}$ of dense subsets of $\PP$ there is a filter $G\subseteq\PP$ such that $G\cap D_\alpha\neq\varnothing$ for all $\alpha<\lambda$. We also say that $G$ is a \emph{$\mathcal D$-generic} filter (of $\PP$).

\begin{definition}The \textbf{$\kappa$-Strongly Proper Forcing Axiom ($\kappa$-$\Str\PFA$)} is \[\FA_{\kappa^+}(\{\PP\mid\PP\text{ is }\kappa\text{-lattice and }\kappa\text{-strongly proper}\}).\]
\end{definition}
We note that unlike the case with $\MA$, where we allow $\cD$ to have any size ${<}2^{\aleph_0}$, here we regard our forcing axiom as an analogue of $\PFA$ and therefore consider only families $\cD$ of size at most $\kappa^+$.
\section{The basic ingredients: supercompact cardinals}\label{section2}
\begin{theorem}\label{thm:supercompact}
Let $\kappa$ be a regular cardinal such that $\kappa^{<\kappa}=\kappa$ and let $\theta>\kappa$ be a supercompact cardinal. Then there is a $\kappa$-lattice and $\kappa$-strongly proper forcing $\PP$ which forces $\theta=\kappa^{++}$ together with $\kappa$-$\Str\PFA$.
\end{theorem}
We prove this theorem by almost entirely repeating the consistency proof of $\PFA$ by finite conditions given by Neeman in \cite{Neeman:2014}, to the point that the authors cannot take credit for this theorem. We will omit most of the proofs of the subclaims, as they are essentially the same as those of Neeman; instead we will indicate, at the appropriate places, what the relevant claims from \cite{Neeman:2014} are. The rest of this section is devoted to the proof of this theorem.

Let $F\colon\theta\to V_\theta$ be a Laver function for $\theta$ and let $E$ denote the set of strong limit cardinals in $\theta$ of cofinality at least $\kappa^+$.
Let $\cS$ be the set of  $\kappa$-closed $M\prec H(\theta)$ such that $|M|=\kappa$ and let $\cT=\{H(\alpha)\mid\alpha\in E\}$. We define for each $\alpha\in E\cup\{\theta\}$ a forcing $\PP_\alpha$ such that $\PP_\alpha$ is a complete subforcing of $\PP_\beta$ for all $\alpha\leq\beta$ in $E\cup\{\theta\}$. Our forcing $\PP$ will be $\PP_\theta$.

Given $\beta\in E\cup\{\theta\}$, we define $\PP_\beta$ as the collection of all the pairs $\tup{p, s}$ such that:
\begin{enumerate}
\item $s\in[(\cS\cup\cT)\cap H(\beta)]^{<\kappa}$ and $\in$ is a weak total order on $s$.
\item $p$ is a function with $\dom(p)\in [E\cap\beta]^{<\kappa}$ such that for all $\alpha\in\dom(p)$,
\begin{enumerate}
\item $F(\alpha)$ is a $\PP_\alpha$-name such that $\forces_\alpha F(\alpha)$ is a $\kappa$-lattice, $\kappa$-strongly proper forcing notion whose conditions are ordinals,\footnote{It is not really necessary to require conditions in the forcing named by $F(\alpha)$ to be ordinals, but it simplifies things to do so.}
\item $H(\alpha)\in s$, and
\item $p(\alpha)$ is a nice $\PP_\alpha$-name such that $\forces_\alpha p(\alpha)\in F(\alpha)$.
\end{enumerate}
\item For every $\alpha\in\dom(p)$ and every $M\in s\cap\cS$ such that $\alpha\in M$, the pair $\tup{p\restriction\alpha, s\cap H(\alpha)}$ is a condition in $\PP_\alpha$ which forces in $\PP_\alpha$ that $p(\alpha)$ is a strong $F(\alpha)$-master condition for $M[\dot G_\alpha]$.
\end{enumerate}

We define $\tup{p_1, s_1}\leq_\beta\tup{p_0, s_0}$ if the following conditions hold:
\begin{itemize}
\item $s_0\subseteq s_1$,
\item $\dom(p_0)\subseteq\dom(p_1)$, and
\item for all $\alpha\in\dom(p_0)$, $\tup{p_1\restriction\alpha, s_1\cap H(\alpha)}\forces_\alpha p_1(\alpha)\leq_{F(\alpha)} p_0(\alpha)$.
\end{itemize}
To simplify the notation, if $\alpha\in E$ and $\tup{p, s}\in\PP_\beta$ for some $\beta>\alpha$, we will write $\tup{p, s}\restriction\alpha$ to denote $\tup{p\restriction\alpha,s\cap H(\alpha)}$.

Given $\beta\in E\cup\{\theta\}$, we denote by $\PP\restriction\beta$ the partial order $\{\tup{p, s}\in\PP\mid\dom(p)\subseteq\beta\}$. Note that there is no restriction on $s$.

\begin{claim}[Claim~6.5 in \cite{Neeman:2014}] Condition (3) in the definition of $\PP_\beta$ is equivalent to, instead of considering (in the hypothesis) those $M$ such that $\alpha\in M\in s\cap\cS$,  considering (in the hypothesis) those $M$ such that $M\in s\cap\cS$ occurs above $H(\alpha)$ in $s$ and such that no model between $H(\alpha)$ and $M$ is transitive.
\end{claim}

\begin{claim}[Claim~6.6 in \cite{Neeman:2014}]
Let $\alpha<\beta$ be two ordinals in $E\cup\{\theta\}$. Suppose that $\tup{p, s}\in\PP\restriction\beta$ with $H(\alpha)\in s$, and let $\tup{q, t}\in\PP\cap H(\alpha)$ be a condition extending $\tup{p, s}\restriction\alpha$. Then $\tup{p, s}$ and $\tup{q, t}$ are compatible in $\PP\restriction\beta$, as witnessed by $\tup{r, s\cup t}$, where $r=q\cup p\restriction[\alpha,\beta)$.
\end{claim}

\begin{claim}[Claim~6.7 in \cite{Neeman:2014}]
Let $\beta\in E\cup\{\theta\}$.
\begin{enumerate}
\item Let $\tup{p, s}\in\PP\restriction\beta$ and $\alpha\in E$ be such that $H(\alpha)\in s$. Then $\tup{p, s}$ is a strongly $H(\alpha)$-generic condition for $\PP\restriction\beta$.
\item Let $\tup{p, s}\in\PP\restriction\beta$ and $\alpha\in E$, and suppose that $\tup{p, s}\in H(\alpha)$. Then $\tup{p, s\cup\{H(\alpha)\}}\in\PP\restriction\beta$.
\item $\PP\restriction\beta$ is strongly proper for $\cT$.
\end{enumerate}
\end{claim}
\begin{claim}[Claim~6.8 in \cite{Neeman:2014}]
Let $\QQ$ be a $\kappa$-lattice $\kappa$-strongly proper forcing. Fix $\lambda$ such that $\QQ\in H(\lambda)$ and an $\in$-chain $\tup{M_i\mid i<\mu}$ of $\kappa$-closed elementary submodels of $H(\lambda)$ with $\QQ\in M_i$, with $\mu<\kappa$. Suppose that $i^*<\mu$ and $q\in M_{i^*}$ is a strongly $M_i$-generic condition for all $i<i^*$. Then there is some $q'\leq q$ which is a strongly $M_i$-generic for all $i<\mu$. In particular, there is a condition $q\in\QQ$ extending any given $q_0\in M_0$ which is strongly $M_i$-generic for all $i<\mu$.
\end{claim}
\begin{proof}[Sketch of Proof of Claim]
We build a decreasing sequence, $q_j$ for $i^*<j<\mu$, of conditions extending $q$ and such that $q_j\in M_j$ is a strongly $M_j$-generic for all $j>i^*$. At limit steps we use the assumption that $\QQ$ is $\kappa$-lattice and take $q_j$ to be the greatest lower bound of $\tup{q_i\mid i<j}$. At successor steps we simply use the fact that there is an extension of $q_j$ to a strongly $M_j$-generic condition $q_{j+1}$.
We apply elementarity to find $q_{j+1}$ in $M_{j+1}$. Finally, by taking $i^*=0$, the last part of the claim follows immediately.
\end{proof}
\begin{claim}[Claim~6.9 in \cite{Neeman:2014}]
Let $\tup{p, s}\in\PP$ such that for some $\alpha$, $H(\alpha)\in s$ but $\alpha\notin\dom(p)$. Moreover, let $M\in s\cap\cS$ and $\tup{q, t}\in\PP\cap M$ be such that $\alpha\in\dom(q)$ and $\tup{p,s}\leq \tup{q\restriction\theta\setminus\{\alpha\},t}$. If $(s\cap M)\setminus H(\alpha)\subseteq t$, then there is a function $p'$ which extends $p$, with $\dom(p')=\dom(p)\cup\{\alpha\}$, and such that $\tup{p', s}\in\PP$ and $\tup{p',s}\leq\tup{q, t}$.
\end{claim}
\begin{claim}[Claim~6.10 in \cite{Neeman:2014}]
Let $\tup{p, s}$ and $\tup{q, t}$ be conditions in $\PP$. Let $M\in s\cap\cS$ such that $\tup{q, t}\in M$. Suppose there is some $\delta<\theta$ such that:
\begin{enumerate}
\item $\tup{p, s}\leq\tup{q\restriction\delta, t}$ and $(\dom(p)\cap\dom(q))\setminus\delta=\varnothing$, and
\item $(s\cap M)\setminus H(\delta)\subseteq t$.
\end{enumerate}
Then there is a function $p'$ extending $p$ such that $\dom(p')=\dom(p)\cup(\dom(q)\setminus\delta)$ and such that $\tup{p', s}\in\PP$ extends $\tup{q, t}$.
\end{claim}
\begin{claim}[Corollary~6.11 in \cite{Neeman:2014}] Let $M\in\cS$ and let $\tup{p, s}\in\PP\cap M$. Then there is condition $\tup{q, t}\in\PP$ that extends $\tup{p, s}$ and is such that $M\in t$.
\end{claim}
\begin{claim}\label{iteration-lattice}
$\PP$ is $\kappa$-lattice.
\end{claim}
\begin{proof}
Suppose that $\tup{\tup{p_i, s_i}\mid i<\mu}$ is a directed system of conditions with $\mu<\kappa$. Let $p$ be the function with domain $\bigcup_{i<\mu}\dom(p_i)$ such that for each $i<\mu$ and $\alpha\in \dom(p_i)$, $p(\alpha)$ is a canonical $\PP_\alpha$-name for a condition forced to be the greatest lower bound of $\{p_j(\alpha)\mid i\leq j<\mu\}$ provided $\{p_j(\alpha)\mid i\leq j<\mu\}$ is a directed set of conditions in $F(\alpha)$. Let also $s$ be the closure of $\bigcup_{i<\mu}s_i$ under intersections. It is then immediate to verify that $\tup{p, s}$ is a lower bound of $\{\tup{p_i, s_i}\mid i<\mu\}$, and it is indeed the greatest lower bound by construction.
\end{proof}
\begin{claim}
Let $\beta\in E$, $\tup{p, s}\in\PP$, and $M\in s\cap\cS$ be such that $\beta\in M$. Suppose that $\tup{p',s'}\in\PP_\beta$ extends $\tup{p, s}\restriction\beta$. Then given any $\tup{\bar p,\bar s}\in\PP_\beta\cap M$ such that $\tup{p',s'}\leq\tup{\bar p,\bar s}$, there is $\pi_M(p',s')\in\PP_\beta\cap M$ such that $\pi_M(p',s')$ extends $\tup{\bar p,\bar s}$ and such that every $\tup{q, t}\in M$ which extends $\pi_M(p',s')$ is compatible with $\tup{p', s'}$. In particular, $\tup{p, s}\restriction\beta$ is a strongly $M$-generic condition for $\PP_\beta$ whenever $\beta\in M\in s$.
\end{claim}
\begin{proof}
The case where $\kappa=\omega$, i.e.\ when we deal with the usual notion of a strongly proper forcing, was proved by Neeman in \cite{Neeman:2014}. We therefore assume $\kappa>\omega$. We prove the claim by induction on $\beta$. Let $\tup{\bar p, \bar s}\in M$ be a condition such that $\tup{p', s'}\leq\tup{\bar p, \bar s}$. Let $\tup{\alpha_i\mid i<\mu}$, for some $\mu<\kappa$, be the strictly increasing enumeration of $\dom(p')\cap M$. Without loss of generality we may assume $\mu>0$, as otherwise the conclusion is immediate.

Using the previous claim that $\PP$ is $\kappa$-lattice and suitable bookkeeping, we may build a $\leq$-decreasing sequence in $\PP$, $\tup{\tup{p_i, s_i}\mid i\leq\mu\cdot\omega}$, where $\tup{p_0,s_0}=\tup{p',s'}$ and for every $i<\mu\cdot\omega$ and $\alpha\in\dom(p_i)\cap M$ there is some $j>i$ such that $\tup{p_j, s_j}\restriction\alpha\in\PP_\alpha$ decides, for some name $\dot{\xi}_j^\alpha\in M$ for an ordinal, that $\dot{\xi}_j^\alpha$ is a condition in $F(\alpha)$ such that every $F(\alpha)$-condition in $M\cap F(\alpha)$ extending $\dot{\xi}^\alpha_j$ is $F(\alpha)$-compatible with $p_j(\alpha)$.

By suitable applications of the induction hypothesis we can make sure that for every $\alpha$ and every increasing sequence of indices $\tup{j_\eta\mid\eta<\nu}$ such that $\dot{\xi}^\alpha_{j_\eta}$ is defined, $\langle \dot{\xi}^\alpha_{j_\eta} \mid\eta<\nu\rangle$ is forced to be a decreasing sequence of conditions in $F(\alpha)$.

Given any $\alpha<\theta$, if a limit stage $i$ of the construction is such that we have dealt with $\alpha$ (i.e., $\dot{\xi}^\alpha_j$ has been defined) cofinally often below $i$, then we let $\dot{\xi}^\alpha_i$ be a $\PP_\alpha$-name for the greatest lower bound of $\{\dot{\xi}^\alpha_j\mid j\in I\}$ in $F(\alpha)$---where $I$ is the cofinal subset of $j\in i$ for which $\dot{\xi}^\alpha_j$ is defined. Since $\dot{\xi}^\alpha_i$ is forced to be the greatest lower bound of $\{\dot{\xi}^\alpha_j\mid j\in I\}$, rather than an arbitrary lower bound of this set, the greatest lower bound of $\{p_j(\alpha)\mid j\in I\}$ is forced to be compatible with $\dot{\xi}^\alpha_i$, and so the construction can keep going. This is the only place where we use the fact that the forcings $F(\alpha)$ are forced to be $\kappa$-lattice, rather than just $\kappa$-closed or even $\kappa$-directed closed.

Let $\tup{p^*, s^*}=\tup{p_{\mu\cdot\omega}, s_{\mu\cdot\omega}}$. We may---and we do---set up our bookkeeping in such a way that $\tup{\bar p^*, s^*\cap M}\in M$ is a condition in $\PP$ extending $\tup{\bar p,\bar s}$. $\tup{\bar p^*,s^*\cap M}$ will be our $\pi_M(p',s')$.

Suppose now that $\tup{q, t}\in M$ is a condition in $\PP$ extending $\tup{\bar p^*,s^*\cap M}$. It is enough to prove that $\tup{q,t}$ is compatible with $\tup{p^*, s^*}$. For this, we let $\tup{\alpha_i\mid i<\mu^*}$, for some $\mu^*<\kappa$, be the strictly increasing enumeration of $\dom(p^\ast)\cup \dom(q)$. We may assume for simplicity that $\mu^*$ is a limit ordinal. We build a decreasing sequence $\langle \tup{q_i, t_i}\mid i\leq\mu^*\rangle$ of $\PP$-conditions such that each $\tup{q_{i+1}, t_{i+1}}$ is a condition in $\PP\restriction\alpha_i$ extending $\tup{p^*\restriction\alpha_i, s^*}$ and $\tup{q\restriction\alpha_i, t}$, taking greatest lower bounds at limit stages. The desired common extension of $\tup{q, t}$ and $\tup{p^*, s^*}$ will be $\tup{q_{\mu^*}, t_{\mu^*}}$. At successor stages $i+1$ for which $\alpha_i\in \dom(q)\cap \dom(p^*)$ we apply the fact that $F(\alpha_i)$ is forced to be $\kappa$-directed closed to find $p_i(\alpha_i)$ which is forced to extend $q_i(\alpha_i)$ and $p^*(\alpha_i)$ in $F(\alpha_i)$, noting that $p^*(\alpha_i)$ is, by construction, a name forced to be the greatest lower bound of a decreasing sequence in $F(\alpha_i)$ of $|i|$-many conditions compatible with $q(\alpha_i)$. This completes the proof.
\end{proof}
To complete the proof of \autoref{thm:supercompact} we observe that the following is a corollary from the above (an analogous corollary appears in \cite{Neeman:2014}).
\begin{corollary}\label{iteration-strongly-proper} For every $\alpha\leq\theta$, $\PP_\alpha$ is $\kappa$-strongly proper.\qed
\end{corollary}
Finally, by standard reflection arguments using the Laver function and the fact that unboundedly often we choose $\Col(\kappa^+,\alpha)$ and $\Add(\kappa,1)$ as $F(\alpha)$, we get the following corollary.
\begin{corollary}$\PP$ forces $\theta=\kappa^{++}$ together with $\kappa$-$\Str\PFA$.\qed
\end{corollary}

The following fact is not needed for the proof of \autoref{thm:supercompact}, but it will be needed for the proofs of \autoref{inacc}, \autoref{mainthm0}, and \autoref{mainthm1}.

\begin{claim}\label{cc} $\PP$ has the $\theta$-c.c.
\end{claim}

\begin{proof}
Let $\tup{p_\xi, s_\xi}\in \PP$ for $\xi<\theta$. By a standard $\Delta$-system argument we may find $I\subseteq$ of size $\theta$, together with $\bar\theta<\theta$, a function $\bar p$, and $\bar s\in [(\cS\cup\cT)\cap H(\bar\theta)]^{<\kappa}$ such that for all $\xi_0<\xi_1$ in $I$:

\begin{enumerate}

\item $\dom(p_{\xi_0})\cap\dom(p_{\xi_1})=\dom(\bar p)$ and $p_{\xi_0}\restriction \bar\theta=p_{\xi_1}\restriction\bar\theta=\bar p$,
\item $s_{\xi_0}\cap H(\bar\theta)=s_{\xi_1}\cap H(\bar\theta)=\bar s$, and
\item there is $\alpha\in E$ such that $\tup{p_{\xi_0}, s_{\xi_0}}\in H(\alpha)$, $p_{\xi_1}\restriction \alpha = \bar p$, and $s_{\xi_1}\cap H(\alpha)=\bar s$.
\end{enumerate}

Let us fix $\xi_0<\xi_1$ in $I$ and let $\alpha\in E$ as in (3) for this pair. We may assume that $s_{\xi_1}\setminus H(\alpha)\neq\varnothing$ as otherwise the proof is simpler. If the member of minimal rank in $s_{\xi_1}\setminus H(\alpha)$ is in $\cT$, then $\tup{p_{\xi_0}\cup p_{\xi_1}, s_{\xi_0}\cup s_{\xi_1}}$ is a condition in $\PP$ extending both $\tup{p_{\xi_0}, s_{\xi_0}}$ and  $\tup{p_{\xi_1}, s_{\xi_1}}$ thanks to the fact that the members of $\cT$ form an $\in$-chain.

Let us now consider the case that the member of minimal rank in $s_{\xi_1}\setminus H(\alpha)$ is a model $M$ in $\cS$. Let $\alpha^*= \min((M\cap\Ord)\setminus\alpha)$. Then $\alpha^*\in E$. To see this, we first note that $\cf(\alpha^*)\geq\kappa^+$. We then note that if $\alpha\notin M$, then $\alpha^*$ is a limit of members of $E$ as $\alpha<\alpha^*$ is above $\sup(M\cap\alpha^*)$, and therefore it is a strong limit.

It now follows that $\tup{p_{\xi_0}\cup p_{\xi_1}, s_{\xi_0}\cup \{H(\alpha^*)\}\cup s_{\xi_1}}$ is a condition in $\PP$  extending both $\tup{p_{\xi_0}, s_{\xi_0}}$ and $\tup{p_{\xi_1}, s_{\xi_1}}$ (again using the fact that the members of $\cT$ form an $\in$-chain).
\end{proof}

\section{Weak reflection of \texorpdfstring{$\kappa$}{k}-strongly proper forcings}
\begin{lemma}[Weak Reflection Lemma]\label{reflection}
Let $\kappa$ be a regular cardinal such that $\kappa^{<\kappa}=\kappa$. Suppose that $\PP$ is $\kappa$-lattice and $\kappa$-strongly proper forcing and let $\cD=\{D_\alpha\mid\alpha<\kappa^+\}$ be a family of dense open sets. Then there is a $\kappa$-lattice and $\kappa$-strongly proper forcing $\PP^*\subseteq\PP$ of size $2^\kappa$ and a family of dense subsets of $\PP^*$, $\cD^*=\{D_\alpha^*\mid\alpha<\kappa^+\}$, such that there is a $\cD$-generic filter of $\PP$  if and only if there is a $\cD^*$-generic filter of $\PP^*$.
\end{lemma}
\begin{proof}
Let $\theta$ be a large enough regular cardinal and let $N\prec H(\theta)$ be a $\kappa^+$-closed elementary submodel such that $\PP, \cD\in N$ and $|N|= 2^\kappa$. Let $\PP^*=\PP\cap N$ and $\cD^*=\{D_\alpha\cap N \mid\alpha<\kappa^+\}$.
\begin{claim}
$\PP^*$ is $\kappa$-lattice and $\kappa$-strongly proper.
\end{claim}
\begin{proof}
The fact that $\PP^*$ is $\kappa$-lattice follows immediately from the closedness of $N$ and elementarity. We now prove that $\PP^*$ is $\kappa$-strongly proper. Let $\lambda\in H(\theta)$ be a large enough regular cardinal, which exists if we choose $\theta$ to be sufficiently large,\footnote{By which we mean $\theta>2^{2^{\kappa}}$, and of course we may assume to have chosen $\theta$ this way.} and let $M\prec H(\lambda)$ be $\kappa$-closed, of cardinality $\kappa$, and such that $\PP\in M$.

By $\kappa^+$-closedness of $N$, we get that $M\cap N\in N$, and of course $|M\cap N|=\kappa$ and $M\cap N$ is $\kappa$-closed. Also, we may assume that $H(\lambda)\in N$, and therefore $M\cap N$ is an elementary submodel of $H(\lambda)$. Since $\lambda$ was large enough, by elementarity of $N$ it follows that whenever $p\in\PP\cap M\cap N$, there is an extension of $p$ to a strongly $M\cap N$-generic condition $q$ for $\PP$. By elementarity, we can find such a $q$ in $N$. But this implies in particular that $q$ is also strongly $M\cap N$-generic for $\PP^*$ (as witnessed by the restriction of the projection function $\pi_M$ to $\PP^*\restriction q$), which of course means that $q$ is strongly $M$-generic for $\PP^*$.
\end{proof}

It is now trivial to see that there is a $\cD^*$-generic filter for $\PP^*$ if and only if there is a $\cD$-generic filter for $\PP$.
\end{proof}

The above lemma should be compared with the well-known fact that if $\PP$ is a c.c.c.\ partial order, $\kappa\leq |\PP|$, and $\{D_\alpha\mid \alpha<\kappa\}$ is a collection of dense subsets of $\PP$, then there is a c.c.c.\ suborder $\QQ$ of $\PP$ such that $|\QQ|=\kappa$ and such that $D_\alpha\cap\QQ$ is a dense subset of $\QQ$ for every $\alpha<\kappa$. This reflection property for c.c.c.\ forcings is of course what enables one to force $\MA_\kappa$, for a given infinite cardinal $\kappa$, without any large cardinals. As we will soon see, the present weak reflection lemma is one of the two main ingredients that will allow us to force $\kappa$-$\Str\PFA$ without any use of large cardinals.

\section{Peeling off supercompactness to inaccessibility}
Given regular cardinals $\kappa<\theta$, we write $S^\theta_\kappa$ to denote $\{\alpha<\theta\mid \cf(\alpha)=\kappa\}$. We define $S^\theta_{{>}\kappa}$ similarly.

\begin{theorem}\label{inacc}
Assume $\GCH$ holds in $V$. Suppose that $\kappa$ is a regular cardinal and $\theta>\kappa$ is an inaccessible cardinal such that $\lozenge(S^\theta_{{>}\kappa})$ holds. Then there is a $\kappa$-lattice and $\kappa$-strongly proper forcing $\PP$ which forces that $\theta=\kappa^{++}=2^\kappa$ and that the $\kappa$-$\Str\PFA$ holds.
\end{theorem}
\begin{proof}
We repeat the same argument as in the proof of \autoref{thm:supercompact} with $\PP=\PP_\theta$ as described in that proof. The main difference is that here we use the diamond sequence to guess the names for our partial orders. To be more precise, we fix a bijection $\varphi\colon\theta\to V_\theta$ and a diamond sequence $\langle A_\alpha\mid \alpha\in\theta, \cf(\alpha)>\kappa\rangle$ on $S^\theta_{{>}\kappa}$, and let $F\colon S^\theta_{{>}\kappa}\to V_\theta$ be the function defined by $F(\alpha)=\varphi``A_\alpha\subseteq V_\theta$ for each $\alpha$. We then proceed as before with this function $F$ in place of the Laver function. It is not difficult to see that all relevant claims from \autoref{section2} apply to the present construction.

Suppose now that $\QQ$ is a $\kappa$-lattice and $\kappa$-strongly proper forcing in $V[G]$, and $\cD$ is a sequence of length $\kappa^+$ of dense open sets. By the weak reflection lemma we can reduce $\QQ$ to $\QQ^*$ of size $\kappa^{++}=\theta=2^\kappa$. Let $\dot{\QQ}^*$ and $\dot{\cD}^*$ be $\PP$-names for $\QQ^*$. Since $\PP$ has the $\theta$-chain condition (\autoref{cc}), we may assume that both $\dot{\QQ}^*$ and $\dot{\cD}^*$ are included in $V_\theta$. By the choice of $F$, there is some large enough $\alpha$ such that $F(\alpha)=\dot\QQ^*\cap V_\alpha$, and for a large enough $\chi$ we can fix $R\prec H(\chi)$ which is $\kappa^+$-closed, $R\cap V_\theta=V_\alpha$, and such that $R$ contains all  the relevant objects. The rest of Neeman's argument will be as before, and hence the proof will be complete, provided we can show that $\PP_\alpha$ forces $F(\alpha)$ to be $\kappa$-lattice and $\kappa$-strongly proper.

The fact that $\forces_\alpha F(\alpha)$ is $\check\kappa$-lattice is straightforward, using that $R$ is $\kappa^+$-closed: Given $\mu<\kappa$ and a sequence $\sigma=\langle \dot r_\alpha\mid \alpha<\mu\rangle$ of $\PP$-names for $\dot{\QQ}^*$-conditions in $R$, $\sigma$ is in $R$, and therefore, by elementarity of $R$ and the fact that $\dot{\QQ}^*$ is forced to be $\kappa$-lattice, we may fix a $\PP$-name in $R$ for a condition which is forced to be the greatest lower bound of $\{\dot r_\alpha\mid \alpha<\mu\}$ provided this set is directed in $\dot{\QQ}^*$.

It remains to prove that $F(\alpha)$ is also forced to be $\kappa$-strongly proper. For this, let $\dot N$ be a $\PP_\alpha$-name for a $\kappa$-closed elementary submodel of some large enough $H(\lambda)$ such that $\forces_\alpha F(\alpha)\in \dot N$ and $|\dot N|=\check\kappa$. We may assume for simplicity that $\lambda\in R$. Let $\dot N'$ be a $\PP_\alpha$-name for $\dot N\cap R[\dot G_\alpha]$, and let $\dot r$ be a $\PP_\alpha$-name for a condition in $F(\alpha)\cap\dot N'$. It suffices to show that there is a name $\dot r^*$, of an extension of $\dot r$ in $F(\alpha)$, forced to be a strongly $\dot N'$-generic condition for $F(\alpha)$.

The key point\footnote{This may be a key point in the present proof, but it is not needed in general for this type of constructions (see \autoref{rem0}).} is that $\dot N'$ may be identified with a $\PP_\alpha$-name $\dot N^\dag$ for a $\kappa$-sequence of ordinals,\footnote{Working in the $\PP_\alpha$-extension of $V$, we may fix an ordinal $\lambda_0$ for which there is a bijection $\varphi\colon H(\lambda)\to\lambda_0$. But then we may identify  $\dot N'$ with an enumeration in length $\kappa$ of $\varphi``\dot N'$.} and since $\PP_\alpha$ is $\kappa$-lattice (by \autoref{iteration-lattice} and the fact $\cf(\alpha)>\kappa$) and $\kappa$-strongly proper (by \autoref{iteration-strongly-proper}), this means that $\dot N^\dag$ may be taken as a $\bar\PP$-name in a complete suborder $\bar\PP$ of $\PP_\alpha$ isomorphic to $\kappa^{<\kappa}$. But by $\bar\PP\subseteq R$ and the $\kappa^+$-closedness of $R$, this means that $\dot N^\dag\in R$ and therefore also $\dot N'\in R$, and since $R\cap V_\theta=V_\alpha$, $R$ thinks that $\dot N'$ is a $\PP$-name for a relevant model. Since $\dot\QQ^*$ is a $\PP$-name of a $\kappa$-strongly proper forcing, the same holds in $R$, and therefore $\dot r$ can be extended to a condition $\dot r^*$ as wanted.
\end{proof}

\section{Reducing the consistency strength to \texorpdfstring{$\ZFC$}{ZFC}}\label{section6}
The next step is to remove the inaccessible cardinal from our hypotheses, thereby arriving at our first main result.
\begin{theorem}\label{mainthm0}
Assume $\GCH$, and let $\kappa<\kappa^+<\theta$ be regular cardinals. Then there is a $\kappa$-lattice and $\kappa$-strongly proper forcing $\PP$ which forces $2^\kappa=\kappa^{++}=\theta$ together with $\kappa$-$\Str\PFA$.
\end{theorem}
Since we can start by forcing with $\Col(\kappa^+, {<}\theta)$, we may as well assume that $\theta=\kappa^{++}$, and that no cardinals are collapsed. More importantly, after this preliminary forcing we may fix a $\lozenge(S^\theta_{\kappa^+})$-sequence $\vec A$.

The proof of the theorem is the same as in the inaccessible case, but we need to find a substitute for the models $H(\alpha)$ ($=V_\alpha$) from the filtration $\langle V_\alpha\mid  \alpha\in E\rangle$ used in the side conditions. For this we simply take a filtration $\vec N=\langle N_\alpha\mid \alpha<\theta\rangle$ of $H(\theta)$ into transitive models such that $N_\alpha$ is $\kappa$-closed for every $\alpha$ of cofinality $\kappa^+$, which we can do thanks to $2^\kappa=\kappa^+$ and $2^{\kappa^+}=\theta$. We then let $E=S^\theta_{\kappa^+}$. We also require that the models $M$  in $\cS$ be such that $M\prec\tup{H(\theta),\in, \vec N, \vec A}$. This way we guarantee that the proof of \autoref{cc} goes through in the present situation.

\section{Getting \texorpdfstring{$\kappa$-$\Str\PFA$}{k-StrPFA} together with \texorpdfstring{$2^\kappa$}{P(k)} large}\label{section7}

In this section we generalise \autoref{mainthm0} by proving that $\kappa$-$\Str\PFA$ is consistent with arbitrarily large values of $2^\kappa$. The theorem is the following.

\begin{theorem}\label{mainthm1}
Assume $\GCH$, and let $\kappa<\kappa^+<\theta$ be regular cardinals. Then there is a $\kappa$-lattice and $\kappa$-strongly proper forcing $\PP$ which forces $2^\kappa=\theta$ together with $\kappa$-$\Str\PFA$.
\end{theorem}

\begin{remark}
$\kappa$-$\Str\PFA$ is the first forcing axiom we know of the form $\FA_{\kappa^+}(\Gamma)$, here $\Gamma=\{\PP\mid\PP\text{ is }\kappa\text{-lattice and }\kappa\text{-strongly proper}\}$, such that $\FA_{\kappa^+}(\Gamma)$ is consistent with $2^\kappa$ arbitrarily large whereas $\FA_{\kappa^{++}}(\Gamma)$ is false. To see that $\FA_{\kappa^{++}}(\Gamma)$ is false it suffices to consider the poset $\PP$ of ${<}\kappa$-sized $\in$-chains of $\kappa$-closed elementary submodels $N\prec H(\kappa^{++})$ such that $|N|=\kappa$. $\PP$ is $\kappa$-lattice and $\kappa$-strongly proper, and an application of $\FA_{\kappa^{++}}(\{\PP\})$ would cover $\kappa^{++}$ by a $\kappa^+$-chain of $\kappa$-sized sets.
\end{remark}

In order to prove \autoref{mainthm1} it will be convenient to actually prove a slightly stronger result. Given a cardinal $\kappa$, a ground model $V_0$, a forcing $\PP\in V_0$, and a $V_0$-generic filter $H\subseteq\PP$ such that $V=V_0[H]$, let us call a forcing notion $\QQ\in V$ \emph{$\kappa$-$V_0$-$H$-strongly proper} in the case that for every large enough cardinal $\theta$ and every $M\in H(\theta)^{V_0}$, if $M\prec H(\theta)^{V_0}$, $|M|^{V_0}=\kappa$, $M$ is $\kappa$-closed in $V_0$, $\PP\in M$, and $\QQ\in M[H]$, then given any $q\in \QQ\in M[H]$ there is an extension of $q$ in $\QQ$ which is strongly $M[H]$-generic for $\QQ$.

We note that this is a more general notion than that of $\kappa$-strong properness, so that the $\FA_{\kappa^+}$ for the class of $\kappa$-lattice posets with this property implies $\kappa$-$\Str\PFA$.

Throughout the following proof of \autoref{mainthm1}, given a forcing notion $\QQ$, $\dot G_\QQ$ will denote the canonical $\QQ$-name for the generic object.
\begin{proof}
We start out by letting $\theta_0=\kappa^{++}$ and fixing, after forcing with $\Col(\kappa^+, {<}\theta_0)$ if necessary, a $\lozenge(S^{\theta_0}_{\kappa^+})$-sequence $\vec A = \tup{A_\alpha\mid\alpha \in S^{\theta_0}_{\kappa^+}}$. Let us call this universe $V$. Our goal will be to build a $\kappa$-lattice and $\kappa$-strongly proper poset $\PP$ forcing $2^\kappa=\theta$ together with $\FA_{\kappa^+}(\Gamma^\kappa_{G_\PP})$, where $\Gamma^\kappa_{G_\PP}$ denotes the class of $\kappa$-lattice forcing notions which are $\kappa$-$V$-$\dot G_\PP$-strongly proper. By the above observation, $\PP$ will then force $\kappa$-$\Str\PFA$.

As in the proof of \autoref{mainthm0}, we fix a filtration $\vec N=\tup{N_\alpha\mid \alpha<\theta_0}$ of $H(\theta_0)$ into transitive models such that $N_\alpha$ is $\kappa$-closed for each $\alpha$ with $\cf(\alpha)=\kappa^+$ and let $E=S^{\theta_0}_{\kappa^+}$.
We consider a sequence $\tup{\PP_\beta\mid\beta\in E\cup\{\theta_0\}}$, built very much as in the construction in \autoref{section6}, except that at each stage $\alpha$ we look at whether $A_\alpha$ codes, not a $\PP_\alpha$-name for a relevant forcing, but a $\PP_\alpha\times \Add(\kappa, \kappa^+)$-name for a forcing which is $\kappa$-lattice and $\kappa$-$V$-$\dot G_{\PP_\alpha\times \Add(\kappa, \kappa^+)}$-strongly proper (and if so, then working parts at $\alpha$ are conditions in this forcing).

Our forcing $\PP$ witnessing \autoref{mainthm1} will now be $\PP_{\theta_0}\times \Add(\kappa, \theta)$. It is clear that $\PP$ forces $2^\kappa=\theta$ and has the $\kappa^{++}$-c.c. The proof that $\PP$ forces $\FA_{\kappa^+}(\Gamma^\kappa_{G_\PP})$ is along the lines of the corresponding proof for \autoref{mainthm0}. Specifically, suppose $\dot\QQ^*$ is a $\PP$-name for a $\kappa$-lattice $\kappa$-$V$-$\dot G_{\PP}$-strongly proper forcing, and for $\alpha<\kappa^+$, $\dot D_\alpha$ is a $\PP$-name for a dense subset of $\dot\QQ^*$.

\begin{claim}
There is a $\PP$-name, $\dot\QQ$, for a $\kappa$-lattice $\kappa$-$V$-$\dot G_{\PP}$-strongly proper suborder of $\dot\QQ^*$ of size $\kappa^{++}$ such that $\dot D_\alpha\cap \dot\QQ$ is dense in $\dot\QQ$ for each $\alpha$.
\end{claim}

\begin{proof}
  For this, let $G_0$ be $V$-generic for $\PP_{\theta_0}$ and let us work in $W=V[G_0]$. Let $\dot R$ be an $\Add(\kappa, \theta)$-name for $\dot\QQ^*$. As in the proof of \autoref{reflection}, let $N\in V$ be a $\kappa^+$-closed (in $V$) elementary submodel of size $2^\kappa=\kappa^{++}$ containing everything relevant (including $\Add(\kappa, \theta)$-names for $\dot D_\alpha$ for each $\alpha<\kappa^+$). In particular, $\PP_{\theta_0}\subseteq N$ and $N[G_0]$ is therefore $\kappa^+$-closed in $W$. Let $G$ be $W$-generic for $\Add(\kappa,\theta)$. We claim that it will suffice to take a name, $\dot\QQ$, for $\QQ=\dot R^G\cap N[G_0][G]$. We also let $H$ be a $V$-generic filter for $\PP$ such that $W[G]=V[H]$, and we will use $\QQ^*$ to denote $(\dot\QQ^*)^H=\dot R^G$.

Using the $\kappa$-closedness of $\Add(\kappa, \theta)$ and the $\kappa^+$-closedness of $N[G_0]$ in $W$ it is easy to see that $\QQ$ is $\kappa$-lattice. To see that it is $\kappa$-$V$-$H$-strongly proper, let $M\in V$, $M\prec H(\lambda)^V$, for large enough $\lambda\in N$, be $\kappa$-closed and of size $\kappa$ in $V$, and such that $\QQ\in M[H]$. Given $q\in\QQ\cap M[H]$, we need to produce a strongly $M$-generic condition for $\QQ$ extending $q$.
As in the proof of \autoref{reflection}, we use the closedness of $N$ in $W$ under $\kappa$-sequences and get that $M\cap N\in N$ is, in $V$, $\kappa$-closed and of size $\kappa$. We then finish as in that proof, noting that any strongly $(M\cap N)[G_0][G]$-generic condition for $\QQ^*$ in $N[G_0][G]$ is a strongly $M[H]$-condition for $\QQ$.
\end{proof}

By $\kappa^{++}$-c.c.\ of $\PP$, we may identify $\dot\QQ$ with a $\PP_{\theta_0}\times \Add(\kappa, \kappa^{++})$-name, which we may code by a subset of $\kappa^{++}$. Now we use our diamond $\vec A$
 to capture $\dot\QQ$ as in the proof of \autoref{mainthm0}.\footnote{When the capturing happens at a stage $\alpha\in E$, we have that $A_\alpha$ codes a $\PP_\alpha\times \Add(\kappa, \alpha)$-name, which we can of course identify with a $\PP_\alpha\times \Add(\kappa, \kappa^+)$-name.}
\end{proof}

\begin{remark}\label{rem0}
\autoref{cohen-subsets}, i.e., the fact that $\kappa$-sequences of ordinals in generic extensions by $\kappa$-lattice $\kappa$-strongly proper forcings belong to $\kappa$-Cohen extensions, is not needed in the proof of \autoref{mainthm1}. This is thanks to the fact that at a stage $\alpha\in E$ in the construction, the models $M$ for which we need to prove strong properness of the relevant forcing come in fact from $V$.
\end{remark}

It is worth pointing out---and follows from a well-known result of Paul Larson---that if $\MM^{++}$ holds and we let $\kappa=\omega_2$, then the construction in \autoref{mainthm0} (and \autoref{mainthm1}) preserves $\MM^{++}$ and so forces $\omega_2$-$\Str\PFA$ ``on top'' of this forcing axiom. And the same thing is of course also true for natural weaker forcing axioms like $\MM$, $\PFA$, and so on.

\section{Relaxing strongness or g.l.b.'s?}\label{relax}
Let $\Gamma_\kappa$ be the class of $\kappa$-lattice $\kappa$-strongly proper posets. As we have seen, while $\kappa$-$\Str\PFA$, i.e.\ $\FA_{\kappa^+}(\Gamma_\kappa)$, is consistent with $\ZFC$, it is too weak to decide the size of $2^\kappa$. In fact, this forcing axiom does not seem to have many applications. It does imply certain weak failures of Club Guessing at $\kappa^+$ (as pointed out in the introduction), as well as  $\mathfrak d(\kappa)>\kappa^+$ and the covering number of natural meagre ideals being greater than $\kappa^+$, but we do not know of any other quotable consequences. In this final section we address the prospect of (mildly) relaxing some of the constraints in the definition of $\Gamma_\kappa$ so as to obtain more powerful forcing axioms.

For the rest of this section, let us fix a regular cardinal $\kappa\geq\omega$ such that $\kappa^{<\kappa}=\kappa$.

Given a model $M$ and a set $X\in M$, let us call $S\subseteq [X]^{\kappa}$ an \emph{$M$-stationary subset of $[X]^{\kappa}$} if for every function $F\colon[X]^{<\omega}\to X$ with $F\in M$ there is some $N\in S\cap M$ such that $F``[N]^{<\omega}\subseteq N$.
This is the natural extension in the $[X]^\kappa$ context of the notion, due to Moore, of $M$-stationarity for collections of countable sets (see \cite{Moore}). Also, let us define the \emph{$\kappa$-Ellentuck topology on $[X]^\kappa$} by declaring basic open sets to be of the form $[s, Y] = \{Z\in [Y]^\kappa\mid s\subseteq Z\}$ for $Y\in [X]^\kappa$ and $s\in [Y]^{<\kappa}$.
We will next generalise Moore's Mapping Reflection Principle ($\MRP$) to the present context.

\begin{definition} $\kappa$-$\MRP$ is the following statement: Let $X$ be a set, let $\theta$ be a cardinal such that $X\in H(\theta)$, and let $\Sigma$ be a function defined on a club of $[H(\theta)]^\kappa$ and such that for every $M\in \dom(\Sigma)$, $\Sigma(M)$ is both an $M$-stationary subset of $[X]^\kappa$ and an open subset of $[X]^\kappa$ in the $\kappa$-Ellentuck topology. Then there is a $\subseteq$-continuous $\in$-chain $\tup{M_\alpha\mid \alpha<\kappa^+}$ of elementary submodels of $H(\theta)$ of size $\kappa$ such that for each $\alpha <\kappa^+$ of cofinality $\kappa$, $M_\alpha$ is $\kappa$-closed and there is some $\bar\alpha<\alpha$ such that $M_\beta\cap X\in \Sigma(M_\alpha)$ for all $\beta\in [\bar\alpha,\alpha)$.
\end{definition}

Thus, $\MRP$ is just $\omega$-$\MRP$. The following fact can be proved by a generalisation of the argument showing that $\MRP$ implies the existence of a well-order of $\power(\omega_1)$ of length $\omega_2$ $\Sigma_1$-definable over $H(\omega_2)$ from any given ladder system on $\omega_1$ and any given $\omega_1$-sequence of pairwise disjoint stationary subsets of $\omega_1$.

\begin{fact}\label{w-order} $\kappa$-$\MRP$ implies that $2^{\kappa^+}=\kappa^{++}$. In fact, a stronger statement is true. Given a club-sequence $\vec C=\tup{C_\alpha\mid\alpha\in S^{\kappa^+}_\kappa}$ and a sequence $\vec S=\tup{S_\xi\mid \xi<\kappa^+}$ of pairwise disjoint stationary subsets of $S^{\kappa^+}_\kappa$, $\kappa$-$\MRP$ implies that there is a well-order of $\power(\kappa^+)$ of length $\kappa^{++}$ which is $\Sigma_1$-definable over $H(\kappa^{++})$ from $\vec C$ and $\vec S$ as parameters.
\end{fact}

\begin{definition}A forcing $\PP$ is \emph{$\kappa$-$\MRP$-strongly proper} if for every large enough $\theta$, every $\kappa$-closed $M\prec H(\theta)$ of size $\kappa$ such that $\PP\in M$, and every $p\in\PP\cap M$ there is $q\leq p$ such that for every $q'\leq_{\PP} q$, \[\mathcal X_{q'}=\{X\in [\PP\cap M]^\kappa\mid\exists \pi_X(q')\in\PP\cap X \forall r\leq_\PP \pi_X(q'), r\in X\to r\compatible q'\}\] is an $M$-stationary subset of $[\PP]^\kappa$.\footnote{The notation $r\compatible q'$ means that $r$ is compatible with $q'$ in $\PP$, that is, they have a joint extension.}
\end{definition}

There is a natural forcing which, for a given open and stationary mapping $\Sigma$ as in the original statement of $\MRP$, adds by finite approximations a reflecting sequence for $\Sigma$.  An immediate generalisation of the proof that $\PFA$ implies $\MRP$ using such forcings yields the following.\footnote{The relevant forcing this time is a natural one for adding a suitable reflecting sequence by ${<}\kappa$-sized approximations.}

 \begin{fact}
$\FA_{\kappa^+}(\{\PP\mid\PP\text{ is }\kappa\text{-lattice and }\kappa\text{-}\MRP\text{-strongly proper}\})$ implies $\kappa$-$\MRP$.
\end{fact}

The bad news is that, as \autoref{incons} shows, this forcing axiom is inconsistent when $\kappa$ is uncountable.

 \begin{theorem}\label{incons} Suppose $\kappa\geq\omega_1$ is such that $\kappa^{<\kappa}=\kappa$. Then \[\FA_{\kappa^+}(\{\PP\mid\PP\text{ is }\kappa\text{-lattice and }\kappa\text{-}\MRP\text{-strongly proper}\})\] is false.
  \end{theorem}

When $\kappa$ is a successor cardinal, one can prove this inconsistency using the following theorem of Shelah (see \cite{Shelah:PIF}, Appendix Chap.\ 3).

\begin{theorem}[Shelah]\label{unif} Let $\kappa\geq\omega_1$ be a regular cardinal and let $\langle C_\alpha\mid\alpha\in S^{\kappa^+}_\kappa\rangle$ be a club-sequence with $\ot(C_\alpha)=\kappa$ for all $\alpha\in S^{\kappa^+}_\kappa$. Then there is a sequence $\langle f_\alpha\mid\alpha\in S^{\kappa^+}_\kappa\rangle$ of colourings, with $f_\alpha\colon C_\alpha\rightarrow\{0, 1\}$ for all $\alpha$, for which there is no function $G\colon \kappa^+\rightarrow 2$ such that for all $\alpha\in S^{\kappa^+}_\kappa$, $G(\xi)=f_\alpha(\xi)$ for club-many $\xi\in C_\alpha$.
  \end{theorem}

 The strategy in this case is to consider $\langle C_\alpha\mid\alpha\in S^{\kappa^+}_\kappa\rangle$ and $\langle f_\alpha\mid\alpha\in S^{\kappa^+}_\kappa\rangle$ as in \autoref{unif} and to apply the forcing axiom to a natural forcing for adding by ${<}\kappa$-sized approximations a regressive function $p$ on $S^{\kappa^+}_\kappa$ such that

 \begin{enumerate}
 \item for all $\alpha\in \dom(p)$, $p(\alpha)<\alpha$, and

\item for all $\alpha_0<\alpha_1$, if $\xi\in (C_{\alpha_0}\setminus p(\alpha_0))\cap (C_{\alpha_1}\setminus p(\alpha_1))$, then $f_{\alpha_0}(\xi)=f_{\alpha_1}(\xi)$.

\end{enumerate}

We then have that  $\langle f_\alpha\mid\alpha\in S^{\kappa^+}_\kappa\rangle$ can be uniformised, in fact modulo co-bounded sets, which is a contradiction.

\begin{remark}
 $\PP$ is also $\kappa^+$-c.c., so this shows the failure of \[\FA_{\kappa^+}(\{\PP\mid\PP\text{ is }\kappa\text{-lattice, }\kappa^+\text{-c.c., and }\kappa\text{-}\MRP\text{-strongly proper}\})\] when $\kappa$ is a successor cardinal such that $\kappa^{<\kappa}=\kappa$.
 \end{remark}

On the other hand, \autoref{unif} does not seem to be available when $\kappa$ is inaccessible. We will now give a proof of \autoref{incons} covering all cases.  This proof uses the following result of Shelah.

\begin{theorem}[Shelah, Claim 3.3 in \cite{Shelah-colouring}]\label{cg} For every uncountable regular cardinal $\kappa$ there is a club-sequence $\tup{C_\alpha\mid \alpha\in S^{\kappa^+}_\kappa}$ such that for every club $D$ of $\kappa^+$ there is some $\alpha\in D$ with $\{\zeta<\kappa \mid C_\alpha(\zeta+1)\in D\}$ stationary (where $\tup{C_\alpha(\zeta)\mid \zeta < \kappa}$ is the strictly increasing enumeration of $C_\alpha$).
 \end{theorem}

 We are now ready to give the proof of \autoref{incons} in the general case.

 \begin{proof}
 Given $\vec C=\langle C_\alpha\mid\alpha\in S^{\kappa^+}_\kappa\rangle$ as in \autoref{cg}, let $\PP$ be the following forcing: Conditions in $\PP$ are pairs $\tup{\mathcal I, b}$ such that:
\begin{enumerate}
\item $\mathcal I$ is a collection of ${<}\kappa$-many pairwise disjoint intervals of the form $[\alpha, \beta]$ with $\alpha\leq\beta<\kappa^+$,
\item $b$ is a regressive function with $\dom(b) \subseteq \{\min(I)\mid I\in\mathcal I\}\cap S^{\kappa^+}_\kappa$,
\item for each $\alpha\in \dom(b)$, $\{\min(I)\mid I\in\mathcal I\}\cap \{C_\alpha(\zeta+1)\mid \zeta<\kappa\}\cap (b(\alpha), \alpha) =\varnothing$, and
\item for each $\alpha\in \dom(b)$ and each $I\in\mathcal I$, if $b(\alpha)<\min(I)<\alpha$ and $I'\in\mathcal I$ is such that $\min(I) < \min(I') < \alpha$, then $\min(I) < C_\alpha(\zeta) < \min(I')$ for some $\zeta$.
 \end{enumerate}

$\tup{\mathcal I_1, b_1} \leq_{\PP} \tup{\mathcal I_0, b_0}$ if

\begin{enumerate}
\item  for every $I\in\mathcal I_0$ there is $I'\in\mathcal I_1$ with $\min(I')=\min(I)$ and $\max(I)\leq\max(I')$ and
\item $b_0\subseteq b_1$.
\end{enumerate}

Then $\PP$ belongs to the relevant class and adds a club of $\kappa^+$ violating the club-guessing property of $\vec C$.
\end{proof}

 \begin{question}
Suppose $\kappa$ is an inaccessible cardinal. Does it necessarily follow that $\FA_{\kappa^+}(\{\PP\mid\PP\text{ is }\kappa\text{-lattice, }\kappa^+\text{-c.c., and }\kappa\text{-}\MRP\text{-strongly proper}\})$ fails?
\end{question}

 Next we will show that the restriction to $\kappa$-strongly proper forcing which, in addition, are $\kappa$-lattice is not just a technical artefact of our consistency proofs but is in fact a necessary restriction. This result is essentially due to Shelah (see Appendix~Chap.~3,~3.4 in \cite{Shelah:PIF}). We include the proof for the reader's convenience.

 \begin{theorem}[Shelah] Let $\kappa$ be a successor cardinal. Then \[\FA_{\kappa^+}(\{\PP\mid\PP\text{ is }\kappa\text{-directed closed, }\kappa^+\text{-c.c., and }\kappa\text{-strongly proper}\})\] is false. \end{theorem}

 \begin{proof}
This is similar to the proof on \autoref{incons} for the case when $\kappa$ is a successor cardinal. Let $\langle C_\alpha\mid\alpha\in S^{\kappa^+}_\kappa \rangle$ be a club-sequence with $\ot(C_\alpha)=\kappa$ for all $\alpha\in S^{\kappa^+}_\kappa$ and let $\tup{f_\alpha\mid\alpha\in S^{\kappa^+}_\kappa}$  be a sequence of colourings, where $f_\alpha\colon C_\alpha\rightarrow\{0, 1\}$ for all $\alpha$, which cannot be club-uniformised in the sense of \autoref{unif}. As in the proof of \autoref{incons} in the successor cardinal case, we will produce a forcing notion $\PP$ adding a uniformizing function and belonging to the relevant class. An application of the forcing axiom to $\PP$ yields then a contradiction.

 Conditions in $\PP$ are pairs $p=\tup{a_p, \vec d_p}$, where
 \begin{enumerate}
 \item $a_p\in [S^{\kappa^+}_\kappa]^{<\kappa}$;
 \item $\vec d_p=\tup{d_p^\alpha\mid\alpha\in a_p}$ is such that, for some successor ordinal $i_p+1<\kappa$, $d_p^\alpha\colon i_p+1\rightarrow C_\alpha$ is a strictly increasing and continuous function;
 \item for all $\alpha_0$, $\alpha_1\in a_p$ and for all $\xi\in \range(d_p^{\alpha_0})\cap\range(d_p^{\alpha_1})$, $f_{\alpha_0}(\xi)=f_{\alpha_1}(\xi)$.
 \end{enumerate}

 The extension relation $\leq$ on $\PP$ is defined by letting $q\leq p$ exactly when
 \begin{enumerate}
 \item $a_p\subseteq a_q$,
 \item for every $\alpha\in a_p$, $d_p^\alpha$ is an initial segment of $d_q^\alpha$, and
 \item if $i_p<i_q$, then $(C_{\alpha_0}\cap C_{\alpha_1})\setminus \min\{d_q^{\alpha_0}(i_q), d_q^{\alpha_1}(i_q)\}=\varnothing$ for all $\alpha_0\neq \alpha_1\in a_p$.
 \end{enumerate}

 Using $\kappa^{<\kappa}=\kappa$ it is straightforward to verify that $\PP$ $\kappa$-directed closed, $\kappa^+$-c.c., and $\kappa$-strongly proper. Hence, an application of the forcing axiom to $\PP$ yields a club-uniformising function for $\tup{f_\alpha\mid\alpha\in S^{\kappa^+}_\kappa}$, which is a contradiction. On the other hand, it is not difficult to see that $\PP$ is not $\kappa$-lattice; in fact, one can easily find compatible conditions $p_0$, $p_1\in\PP$ which do not have a greatest lower bound.
 \end{proof}

The following is now a natural question.

 \begin{question} Suppose $\kappa$ is an inaccessible cardinal. Does it necessarily follow that $\FA_{\kappa^+}(\{\PP\mid\PP\text{ is }\kappa\text{-directed closed and }\kappa\text{-strongly proper}\})$ fails?
\end{question}

 We will finish with the following question.

 \begin{question}
 Is it consistent that there is any uncountable regular $\kappa$ for which $\kappa$-$\MRP$ holds? More generally, and in view of \autoref{w-order}, is there any $\Pi_2$ sentence $\sigma$ with the following properties?

 \begin{enumerate}
 \item $\ZFC$ proves that if $\kappa\geq\omega_1$ is a regular cardinal and $H(\kappa^{+})\models\sigma$ holds, then $2^{\kappa}=\kappa^+$.
 \item Some reasonable extension of $\ZFC$ proves that one can force the existence of a regular cardinal $\kappa\geq\omega_1$ such that $H(\kappa^+)\models\sigma$.
 \end{enumerate}
 \end{question}

\bibliographystyle{amsplain}
\providecommand{\bysame}{\leavevmode\hbox to3em{\hrulefill}\thinspace}
\providecommand{\MR}{\relax\ifhmode\unskip\space\fi MR }
\providecommand{\MRhref}[2]{%
  \href{http://www.ams.org/mathscinet-getitem?mr=#1}{#2}
}
\providecommand{\href}[2]{#2}

\end{document}